\documentclass[reqno]{amsart}

\usepackage{mathrsfs}
\usepackage{amsmath}
\usepackage{amssymb}
\usepackage{cite}
\usepackage{latexsym}
\usepackage{graphicx}

\usepackage{color}
\usepackage{comment}

\theoremstyle{plain}
\newtheorem{thm}{Theorem}[section]
\newtheorem{lem}[thm]{Lemma}

\newtheorem{prop}[thm]{Proposition}
\newtheorem{rmk}[thm]{Remark}

\newtheorem{ex}[thm]{Example}

\def\G{\mathscr{G}}

\def\P{\mathscr{P}}
\def\S{\mathscr{S}}
\def\T{\mathscr{T}}

\def\d{\mathrm{d}}

\def\l{\mathrm{L}}
\def\r{\mathrm{R}}

\def\Cset{\mathbb{C}}
\def\Eset{\mathbb{E}}
\def\Kset{\mathbb{K}}
\def\Lset{\mathbb{L}}
\def\Nset{\mathbb{N}}
\def\Rset{\mathbb{R}}

\def\Zset{\mathbb{Z}}

\def\GL{\mathrm{GL}}
\def\SL{\mathrm{SL}}
\def\gl{\mathrm{gl}}

\def\id{\mathrm{id}}
\def\Im{\mathrm{Im}\,}
\def\Re{\mathrm{Re}\,}

\def\epsilon{\varepsilon}

\DeclareMathOperator{\diag}{diag}
\DeclareMathOperator{\sech}{sech}

\makeatletter
 \@addtoreset{equation}{section}
\makeatother
\def\theequation{\arabic{section}.\arabic{equation}}

\begin{document}


\title[Solvability of the KdV equation by quadrature]
{Solvability of the Korteweg-de Vries equation
 under meromorphic initial conditions by quadrature}

\author{Kazuyuki Yagasaki}

\address{Department of Applied Mathematics and Physics, Graduate School of Informatics,
Kyoto University, Yoshida-Honmachi, Sakyo-ku, Kyoto 606-8501, JAPAN}
\email{yagasaki@amp.i.kyoto-u.ac.jp}

\date{\today}

\begin{abstract}
We study the solvability of the Korteweg-de Vries equation
 under meromorphic initial conditions by quadrature
 when the inverse scattering transform (IST) is applied.
It is a key to solve the Schr\"odinger equation appearing in the Lax pair
 in application of the IST.
We show that the Schr\"odinger equation is always integrable
 in the sense of differential Galois theory, i.e., solvable by quadrature,
 if and only if the meromporphic potential is reflectionless,
 under the condition that  the potential is absolutely integrable on $\Rset\setminus(-R_0,R_0)$ for some $R_0>0$.
This statement was previously proved to be true by the author
 for a limited class of potentials.
We also show that the Schr\"odinger equation is not integrable in this sense
 for rational potentials that decay at infinity but do not satisfy  the weak condition.
\end{abstract}
\maketitle


\section{Introduction}

In this paper we study the solvability of the Korteweg-de Vries (KdV) equation
\begin{equation}
u_t+6uu_x+u_{xxx}=0
\label{eqn:KdV}
\end{equation}
under meromorphic initial conditions by quadrature,
 where the subscripts represent differentiation
 with respect to the independent variables.
Here the solution to \eqref{eqn:KdV} are required to satisfy
\[
\lim_{x\to\pm\infty}u(x,t)=0.
\]
We also pay special attention to the case
 in which the initial conditions are given by rational functions.

As well-known, the initial value problems (IVPs) of the KdV equation \eqref{eqn:KdV}
 can be solved by the inverse scattering transform (IST)
 (see, e.g., Chapter~9 of \cite{A11}).
In the method, a pair of differential equations 
\begin{align}
&
\varphi_{xx}+u(x,t)\varphi=\lambda\varphi,
\label{eqn:LP1}\\
&
\varphi_t=-4\varphi_{xxx}-6u(x,t)\varphi_x+(\mu-3u_x(x,t))\varphi,
\label{eqn:LP2}
\end{align}
which are called the \emph{Lax pair}, are analyzed to obtain the solutions to the IVP,
 where $\lambda,\mu\in\Cset$ are constants.
In particular, we need to solve \eqref{eqn:LP1}
 for any $\lambda\in\Cset^\ast=\Cset\setminus\{0\}$
 and obtain particular solutions called the \emph{Jost solutions}, which satisfy
\begin{equation}
\begin{split}
&
\phi(x;k)\sim e^{-ikx}\quad\mbox{as $x\to-\infty$},\\
&
\psi(x;k)\sim e^{ikx}\quad\mbox{as $x\to+\infty$},
\end{split}
\label{eqn:bc}
\end{equation}
where $\lambda=-k^2\neq 0$.
Thus, when the IST is applied, the IVP of \eqref{eqn:KdV} is solvable by quadrature
 only if so is \eqref{eqn:LP1} for any $\lambda\in\Cset$.
We discuss the solvability of \eqref{eqn:KdV} by quadrature in this meaning,
 reducing it to that of the Schr\"odinger equation
\begin{equation}
v_{xx}+(k^2+u(x))v=0,
\label{eqn:ZS0}
\end{equation}
or as a first-order system
\begin{equation}
w_x=
\begin{pmatrix}
0 & 1\\
-k^2-u(x) & 0
\end{pmatrix}w,\quad
w=
\begin{pmatrix}
v\\
\d v/\d x
\end{pmatrix}
\in\Cset^2,
\label{eqn:ZS}
\end{equation}
i.e., its integrability in the sense of differential Galois theory \cite{CH11,PS03}.
See also Section~3 of \cite{Y23}
 for a quick review of a necessary part of the theory.

Such integrability of \eqref{eqn:ZS0} (and equivalently \eqref{eqn:ZS})
 was recently discussed in \cite{Y23} under the following condition:
\begin{enumerate}
\setlength{\leftskip}{-3.8mm}
\item[\bf(A${}_0$)]
The potential $u(x)$ is holomorphic in a neighborhood $U$ of $\Rset$ in $\Cset$.
Moreover, there exist holomorphic functions $u_\pm:U_0\to\Cset$
 such that $u_\pm(0)=0$ and
\[
u(x)=u_\pm(e^{\mp\lambda_\pm x})
\]
for $|\Re x|$ sufficiently large,
 where $U_0$ is a neighborhood of the origin in $\Cset$,
 $\lambda_\pm\in\Cset$ are some constants with $\Re\lambda_\pm>0$,
 and the upper or lower signs are taken simultaneously
 depending on whether $\Re x>0$ or $\Re x<0$.
\end{enumerate}
In particular, $u(x)$ tends to zero exponentially as $x\to\pm\infty$
 on $\Rset$, so that $u\in L^1(\Rset)$, if it satisfies condition~(A${}_0$).
Condition~(A${}_0$) is a little restrictive,
 but it is satisfied by several wide classes of functions.
For example, if $u(x)$ is a rational function of $e^{\lambda x}$
 for some $\lambda\in\Cset$ with $\Re\lambda>0$,
 has no singularity on $\Rset$, and $u(x)\to 0$ as $x\to\pm\infty$,
 then condition~(A${}_0$) holds.
See Section~1 of \cite{Y23} for another example of $u(x)$ satisfying condition~(A${}_0$).

Define the \emph{scattering coefficients} $a(k)$ and $b(k)$ for \eqref{eqn:ZS0} as
\begin{equation}
\phi(x;k)=a(k)\psi(x;-k)+b(k)\psi(x;k).
\label{eqn:ab+}
\end{equation}
When $a(k)\neq 0$, the constant $\rho(k)=b(k)/a(k)$
  is defined and called the \emph{reflection coefficient}.
The potential $u(x)$ is called \emph{reflectionless}
 if $b(k)=0$ for any $k\in\Rset^\ast:=\Rset\setminus\{0\}$.
The following result was proved in \cite{Y23}.

\begin{thm}
\label{thm:Y}
Suppose that condition~{\rm(A${}_0$)} holds.
If the potential $u(x)$ is reflectionless,
 then Eq.~\eqref{eqn:ZS0} is solvable by quadrature
 for any $k\in\Cset^\ast$.
Conversely, 
 if Eq.~\eqref{eqn:ZS0} is solvable by quadrature for any $k\in\Rset^\ast$,
 then the potential $u(x)$ is reflectionless.
 \end{thm}

\begin{rmk}
A similar statement was proved
 for more general two-dimensional systems called
  the \emph{Zakharov-Shabat $($ZS$)$ systems},
 which appear in application of the IST for other integrable partial differential equations $($PDEs$)$
 including the nonlinear Schr\"odinger equation,
 modified KdV equation, sine-Gordon equation and sinh-Gordon equation,
 in {\rm\cite{Y23}}.
\end{rmk}

In this paper we consider the case in which $u(x)$ is meromorphic,
 and extend the result of \cite{Y23}. 
More precisely, we assume the following.

\begin{enumerate}
\setlength{\leftskip}{-3.8mm}
\item[\bf(A${}_1$)]
The potential $u(x)$ is meromorphic in a neighborhood of $\Rset$ in $\Cset$ and
\begin{equation}
\int_{\Rset\setminus(-R_0,R_0)}|u(x)|\d x<\infty\quad
\mbox{for some $R_0>0$.}
\label{eqn:A1}
\end{equation}
\end{enumerate}

When Eq.~\eqref{eqn:A1} holds, we have
\begin{equation}
\lim_{x\to\pm\infty}u(x)=0,
\label{eqn:A1a}
\end{equation}
and by Theorem~8.1 in Section~3.8 of \cite{CL55}
 there exist the Jost solutions $\phi(x;k),\psi(x;k)$ satisfying \eqref{eqn:bc}
 for $k\in\Cset^\ast$. 
See also Section~2 of \cite{Y23}.
On the other hand, meromorphic functions that are not analytic have singular points.
Singular solitons called \emph{positons}, \emph{negatons} and \emph{complexitons}
 in the KdV equation \eqref{eqn:KdV} have attracted much attention and been studied
 (see, e.g., \cite{KP99,M02,MY04,Mat02,MS91,RSK96}).
Moreover, rational solitons, which have singularities, in the KdV equation \eqref{eqn:KdV}
 were studied in \cite{AM78,AMM77,C06b,JMSZ19}.
In particular, Jim\'enez et al. \cite{JMSZ19}
 discussed the integrability of the Schr\"odinger equation \eqref{eqn:ZS0}
 in the sense of differential Galois theory
 for the Adler-Moser rational potentials \cite{AM78,AMM77}, for example,
\begin{equation}
u(x)=\frac{2}{x^2},\
\frac{6x(x^3-6\tau)}{(x^3+3\tau)^2},\
\frac{6x(2x^9+675\tau^2 x^3+1350\tau^3)}{(x^6+15\tau x^3-45\tau^2)^2},\ \ldots,
\label{eqn:AM}
\end{equation}
where $\tau\in\Cset$ is any constant, which correspond to rational solitons in the KdV equation \eqref{eqn:KdV}.
See \cite{AM78,JMSZ19} for more details of the Adler-Moser rational potentials
 including their precise definition. 
They showed that its differential Galois group is trivial for $k=0$
 and diagonalizable for $k\neq 0$.
This result implies that Eq.~\eqref{eqn:ZS0} is solvable by quadrature
 for the Adler-Moser rational potentials.
We also notice that Acosta-Hum\'anez et al. \cite{AMW11}
 and Bl\'azquez-Sanz and Yagasaki \cite{BY12}
 studied the integrability of the Schr\"odinger equation \eqref{eqn:ZS0}
 for rational potentials and potentials satisfying condition~{\rm(A${}_0$)},
 respectively, in the sense of differential Galois theory.
 
Let $\Cset_\pm=\{k\in\Cset\mid\pm\mathrm{Im}\,k>0\}$.
Our main results are stated as follows.

\begin{thm}
\label{thm:main1}
Suppose that condition~{\rm(A${}_1$)} holds and $a(k)$ has a zero in $\Cset_+$.
If $b(k)=0$ for any $k\in\Rset^\ast$,
 then Eq.~\eqref{eqn:ZS0} is solvable by quadrature for any $k\in\Cset^\ast$.
\end{thm}

\begin{thm}
\label{thm:main2}
Suppose that condition~{\rm(A${}_1$)} holds
 and $u(x)$ is analytic in a neighborhood of $x=\infty$
 in the Riemann sphere $\Cset\cup\{\infty\}$.
If Eq.~\eqref{eqn:ZS0} is solvable by quadrature for any $k\in\Cset^\ast$,
 then $b(k)=0$ for any $k\in\Rset^\ast$.
\end{thm}

\begin{rmk}\
\label{rmk:1b}
\begin{itemize}
\setlength{\leftskip}{-1.6em}
\item[\rm(i)]
The scattering coefficient $a(k)$ may have no zero in $\Cset_+$
 when condition {\rm(A${}_1$)} holds and $b(k)=0$ for any $k\in\Rset^\ast$.
Actually, for the Adler-Moser rational potentials \eqref{eqn:AM},
 condition {\rm(A${}_1$)} holds
 but $a(k)=1$ and $b(k)=0$ for any $k\in\Cset^\ast$,
 as shown in {\rm\cite{JMSZ19}}.
In such a case, Theorem~$\ref{thm:main1}$ does not apply.
\item[\rm(ii)]
Statements similar to those of Theorems~$\ref{thm:main1}$ and $\ref{thm:main2}$
 are proven for the ZS systems in {\rm\cite{Y25}}.
\end{itemize}
\end{rmk}

Furthermore, we consider the case in which $u(x)$ is a rational function
 and satisfies \eqref{eqn:A1a} but does not \eqref{eqn:A1}.
 
\begin{thm}
\label{thm:main3}
Suppose that $u(x)$ is a rational function,
 and let $m_1$ and $m_2$ be the degrees of its numerator and denominator.
If $m_2-m_1=1$,
 then Eq.~\eqref{eqn:ZS0} is not solvable by quadrature
 for some $k\in\Rset^\ast$.
\end{thm}

Obviously, if $m_2-m_1>1$, then the rational function $u(x)$ satisfies \eqref{eqn:A1}
 and condition~{(A${}_1$).

The outline of this paper is as follows.
In Section~2 we collect preliminary results
 for proving Theorems~\ref{thm:main1} and \ref{thm:main2}.
We give proofs of Theorems~\ref{thm:main1}, \ref{thm:main2} and \ref{thm:main3}
 in Sections~3, 4 and 5, respectively.
Fundamental information on local differential Galois groups near irregular singularities,
 which is needed in Section~4, is provided in Appendix~A,
 and the result of Kovacic \cite{K86}, which plays a crucial role in Section~5,
 is briefly reviewed in Appendix~B.
 

\section{Preliminary Results}
In this section we give preliminary results
 required to prove Theorems~\ref{thm:main1} and \ref{thm:main2}.
Henceforth we assume that condition~(A${}_1$) holds.

Let $v_l(x)$, $l=1,2$, be scalar functions on $\Rset$ or $\Cset$,
 and let $W(v_1(x),v_2(x))$ denote the Wronskian of $v_1(x)$ and $v_2(x)$, i.e.,
\begin{equation}
W(v_1(x),v_2(x))=\begin{vmatrix}
v_1(x) & v_2(x)\\
v_{1x}(x) & v_{2x}(x)
\end{vmatrix}
=v_1(x)v_{2x}(x)-v_2(x)v_{1x}(x).
\label{eqn:2a}
\end{equation}
Since the first-order derivative $v_x$ does not appear in \eqref{eqn:ZS0},
  the Wronskians of $\phi(x;-k),\phi(x;k)$ and of $\psi(x;k),\psi(x;-k)$ are independent of $x$,
  so that by \eqref{eqn:bc}
\begin{equation}
W(\phi(x;k),\phi(x;-k))=W(\psi(x;-k),\psi(x;k))=2ik.
\label{eqn:W}
\end{equation}
Replacing $k$ with $-k$ in \eqref{eqn:ab+}, we have
\begin{equation}
\phi(x;-k)=a(-k)\psi(x;k)+b(-k)\psi(x;-k),
\label{eqn:ab-}
\end{equation}
and consequently
\begin{equation}
a(k)a(-k)-b(k)b(-k)=1
\label{eqn:2b}
\end{equation}
by \eqref{eqn:W}.
We have the following properties of the scattering coefficients. 

\begin{prop}\
\label{prop:2a}
\begin{itemize}
\setlength{\leftskip}{-1.6em}
\item[(i)]
$a(k),b(k)$ are analytic in $\Rset^\ast;$
\item[(ii)]
$a(k)$ can be analytically continued in $\Cset_+;$
\item[(iii)]
$a(k)$ only has discrete zeros in $\Cset_+\cup\Rset;$
\item[(iv)]
If $b(k)=0$ for any $k\in\Rset^\ast$,
 then $a(k)$ only has finitely many zeros in $\Cset_+$.
\end{itemize}
\end{prop}

Similar results are obtained in Section~4.1
 (especially Proposition~4.1) of \cite{Y23}
 when $u(x)$ is analytic on $\Rset$.

\begin{proof}
Let $S$ denote the set of poles of $u(x)$.
Then $S$ is discrete and contains no accumulation point.
Choose a curve $\Gamma=\{\gamma(\xi)\in\Cset_+\mid \xi\in\Rset\}$ with
\[
\gamma(\xi)=\xi+ic\sech\xi
\]
does not intersect $S$.
Hence, $u(x)$ is analytic on $\Gamma$.

Regard \eqref{eqn:ZS0} as a differential equation on $\Gamma$
 and let $U$ be a neighborhood of $\Gamma$ in $\Cset_+$ such that $S\cap U=\emptyset$.
Then its solutions $v=\phi(x;k)$ and $\psi(x;k)$ are bounded and analytic
 in $k\in\Rset^\ast$ as well as in $x\in U$.
By \eqref{eqn:ab+} and \eqref{eqn:2a} we have
\begin{equation}
a(k)=\frac{W(\phi(x,k),\psi(x;k))}{2ik},\quad
b(k)=-\frac{W(\phi(x,k),\psi(x;-k))}{2ik}.
\label{eqn:prop2a1}
\end{equation}
Hence, we obtain part~(i).
Similarly we see that $v=\phi(x;k)$ and $\psi(x;k)$ are bounded and analytic
 in $\Cset_+$ for $x\in U$ satisfying $|x|<R$ with some $R>0$.
This yields part~(ii) along with \eqref{eqn:prop2a1}.

Equation~\eqref{eqn:ZS0} is written as
\begin{equation}
v_{\xi\xi}-\frac{\gamma_{\xi\xi}(\xi)}{\gamma_\xi(\xi)}v_\xi
 +\gamma_\xi(\xi)(k^2+u(\gamma(\xi)))v=0
\label{eqn:prop2a2}
\end{equation}
on $\Gamma$.
Changing the independent variable as $\xi\mapsto k\xi$ in \eqref{eqn:prop2a2}, we have
\[
v_{\xi\xi}-\frac{\gamma_{\xi\xi}(\xi/k)}{k\gamma_\xi(\xi/k)}v_\xi
 +\gamma_\xi(\xi/k)\left(1+\frac{1}{k^2}u(\gamma(\xi/k))\right)v=0,
\]
which reduces to
\[
v_{\xi\xi}+v=0
\]
as $k\to\infty$.
This implies that
\[
a(k)\to 1,\quad b(k)\to 0\quad\mbox{as $|k|\to\infty$}.
\]
We apply the identity theorem (e.g., Theorem~3.2.6 of \cite{AF03})
 to obtain  part~(iii), since $a(k)$ is analytic in $\Cset_+\cup\Rset$.
If $b(k)=0$ for any $k\in\Rset^\ast$,
 then by its analyticity $a(k)$ has no zero near $k=0$ in $\overline{\Cset}_+$.
Hence, part~(iv) follows from part~(iii).
\end{proof}

\begin{rmk}
The conclusion in Proposition~{\rm\ref{prop:2a}(iv)} holds
 under a weaker condition.
For example, if $\lim_{k\to 0} b(k)b(-k)=b_0$ $(\neq -1)$, then it holds.
\end{rmk}

Let $G$ be the differential Galois group of \eqref{eqn:ZS0}.
Then $G$ is an algebraic group such that $G\subset\SL(2,\Cset)$
 since Eq.~\eqref{eqn:ZS0} has no term depending on $v_x$
 (see, e.g., Chapter~6 of \cite{K76} or Section 2.2 of \cite{M99}).
We have the following classification
 for such algebraic groups (see Section~2.1 of \cite{M99} for a proof).
Recall that an algebraic group $\G$ contains in general a unique maximal algebraic subgroup $\G^0$,
 which is called the \emph{connected component of the identity}
 or \emph{connected identity component}.

\begin{prop}
\label{prop:2b}
Any algebraic group $\G\subset\SL(2,\Cset)$ is similar to one of the following types:
\begin{enumerate}
\setlength{\leftskip}{-1.2em}
\item[(i)] $\G$ is finite and $\G^0= \{\id\}$,
 where $\id$ is the $2\times 2$  identity matrix$;$
\item[(ii)] $\G = \left\{
\begin{pmatrix}
\lambda & 0\\
\mu & \lambda^{-1} 
\end{pmatrix}
\middle|\,\lambda\text{ is a root of $1$, $\mu\in\Cset$}
\right\}$
and $\G^0 = \left\{\begin{pmatrix}
1&0\\
\mu & 1 
\end{pmatrix}
\middle|\,\mu \in \Cset\right\}$;
\item[(iii)] 
$\G = \G^0 = 
\left\{
\begin{pmatrix}
\lambda &0 \\
0 & \lambda ^{-1}
\end{pmatrix}
\middle|\,\lambda \in \Cset^{*}
\right\}$;
\item[(iv)]
$\G = \left\{
\begin{pmatrix}
\lambda & 0\\
0 & \lambda^{-1} 
\end{pmatrix},
\begin{pmatrix}
0 & -\beta^{-1}\\
\beta & 0 
\end{pmatrix}
\middle|\,\lambda, \beta \in \Cset^{*}
\right\}$
and $\G^0 = \left\{\begin{pmatrix}
\lambda &0\\
0 & \lambda^{-1}
\end{pmatrix}
\middle|\, \lambda \in \Cset^*\right\}$;
\item[(v)] $\G = \G^0 = \left\{
\begin{pmatrix}
\lambda & 0\\
\mu & \lambda^{-1}
\end{pmatrix}
\middle|\,\lambda \in \Cset^{*},\, \mu \in \Cset
\right\}$;
\item[(vi)] $\G = \G^0 = \SL (2,\, \Cset)$.
\end{enumerate}
\end{prop}

The system~\eqref{eqn:ZS0} is solvable by quadrature unless $G$ is of type (vi).
Proposition~\ref{prop:2b} plays a key role
 in the proof of Theorem~\ref{thm:main2} in Section~4.


\section{Proof of Theorem~\ref{thm:main1}}
In this section we give a proof of Theorem~\ref{thm:main1}.
Henceforth we assume that the hypotheses of Theorem~\ref{thm:main1} hold
 and $b(k)=0$ for $k\in\Rset^\ast$.

\begin{proof}[Proof of Theorem~$\ref{thm:main1}$]
We see by Proposition~\ref{prop:2a}(iii) and \eqref{eqn:2b}
 that $a(k)\neq 0$ for $k\in\Rset^\ast$
 and $a(k)$ has finitely many zeros in $\Cset_+$.
Assume that $a(k)$ has $n$ zeros and
 let $k_j$ be its zero of multiplicity $\nu_j$ in $\Cset_+$ for $j=1,\ldots,n$.
Let the curve $\Gamma=\{\gamma(\xi)\in\Cset_+\mid \xi\in\Rset\}$ contain no pole of $u(x)$,
 and let $U$ be a neighborhood of $\Gamma$ in $\Cset_+$ containing no pole of $u(x)$,
 as in the proof of Proposition~\ref{prop:2a}.

Let
\begin{equation}
M(x;k)=\phi(x;k)e^{ikx},\quad
N(x;k)=\psi(x;k)e^{ikx}.
\label{eqn:MN1}
\end{equation}
By \eqref{eqn:bc} we have
\begin{equation}
\mbox{$M(x;k)\sim 1$ as $x\to-\infty$},\quad
\mbox{
$N(x;-k)\sim 1$ as $x\to+\infty$.}
\label{eqn:MN1a}
\end{equation}
Since $\phi(x;k)$ and $\psi(x;-k)$ are analytic in $x\in U$ and $k\in\Cset_+$,
 so are $M(x;k)$ and $N(x;-k)$.
Using \eqref{eqn:ab+}, we have
\begin{equation}
M(x;k)=a(k)N(x;-k)e^{2ikx}+b(k)N(x;k).
\label{eqn:MN1b}
\end{equation}
Let
\[
N_j^r(x):=\frac{\partial^rN}{\partial k^r}(x;k_j),\quad
r=0,\ldots,\nu_j-1,\quad
j=1,\ldots,n.
\]
Differentiating \eqref{eqn:MN1b} with respect $k$ and substituting $k=k_j$, we obtain
\begin{equation}
\frac{\partial^rM}{\partial k^r}(x;k_j)
=\frac{\partial^r}{\partial k^r}\bigl(b(k)N(x;k)\bigr)\Big|_{k=k_j},\quad
r=0,\ldots,\nu_j-1,
\label{eqn:MN1c}
\end{equation}
for $j=1,\ldots,n$ since $a(k)$ has a zero of multiplicity $\nu_j$ at $k=k_j$.
In particular, the right hand sides of \eqref{eqn:MN1c}
 can be represented by linear combinations of $N_j^r(x)$, $r=0,\ldots,\nu_j-1$.
Actually, when $\nu_j>1$, we have
\begin{align*}
M(x;k_j)
=b(k_j)N_j^0(x),\quad
\frac{\partial M}{\partial k}(x;k_j)
=b(k_j)N_j^1(x)+b_k(k_j)N_j^0(x),\quad
\ldots.
\end{align*}

Define the projection operators
\[
\P^\pm(f)=\frac{1}{2\pi i}\int_{-\infty}^\infty\frac{f(\kappa)}{\kappa-(k\pm i 0)}\d\kappa
\]
for $f\in L^1(\Rset)$, where $k\in\Cset_+$ or $\Cset_-$
 depending on whether the upper or lower signs are taken simultaneously.
If $f_+$ (resp. $f_-$) is analytic in $\Cset_+$ (resp. in $\Cset_-$)
 and $f_\pm(k)\to 0$ as $|k|\to\infty$, then
\[
\P^\pm(f_\pm)=\pm f_\pm,\quad
\P^\pm(f_\mp)=0
\]
(see, e.g., Chapter~7 of \cite{AF03}).
Dividing both sides of \eqref{eqn:MN1b} by $a(k)$ and noting \eqref{eqn:MN1a},
 we apply $\P^-$ to the resulting equation to obtain
\begin{equation}
N(x;k)e^{-2ikx}=1-\frac{1}{2\pi i}\sum_{j=1}^n\int_{|\kappa-k_j|=\delta}
 \frac{M(x;\kappa)}{(k+\kappa)a(\kappa)}\d\kappa
\label{eqn:MN1d}
\end{equation}
since $M(x;k)$ and $N(x;-k)$ are analytic for $k\in\Cset_+$ and $\Cset_-$, respectively,
 and $b(k)=0$ for $k\in\Rset^\ast$.
Using the method of residues, we see that the above integrals
 are represented by linear combinations of $N_j^r(x)$, $r=0,\ldots,\nu_j-1$.
For example, when $\nu_j=1$,
\[
\frac{1}{2\pi i}\int_{|\kappa-k_j|
 =\delta}\frac{M(x;\kappa)}{(k+\kappa)a(\kappa)}\d\kappa
 =\frac{b(k_j)}{(k+k_j)a_k(k_j)}N_j^0(x),
\]
and when $\nu_j=2$,
\begin{align*}
&
\frac{1}{2\pi i}\int_{|\kappa-k_j|=\delta}\frac{M(x;\kappa)}{(k+\kappa)a(\kappa)}\d\kappa\\
&
=\frac{2}{(k+k_j)a_{kk}(k_j)}
 \biggl(b(k_j)N_j^1(x)
 +\biggl(b_k(k_j)-\frac{b(k_j)}{k+k_j}
 -\frac{a_{kkk}(k_j)b(k_j)}{3a_{kk}(k_j)}\biggr)N_j^0(x)\biggr).
 \end{align*}

Differentiating \eqref{eqn:MN1d} with respect to $k$ up to $\nu_j-1$ times
 and setting $k=k_j$,
 we obtain a system of linear equations
 about $N_j^r(x)$, $r=0,\ldots,\nu_j-1$, $j=1,\ldots,n$,
 and solve it to obtain them as rational functions
 of $x$ and $\{e^{2ik_jx}\}_{j=1}^n$, using basic arithmetic operations.
Note that
\[
\frac{\partial}{\partial k}(N(x;k)e^{-2ik x})
=\left(\frac{\partial N}{\partial k}(x;k)-2ixN(x;k)\right)e^{-2ik x}
\]
and so on.
Hence, we see by \eqref{eqn:MN1} and \eqref{eqn:MN1d}
 that the Jost solution $\psi(x;k)$
 is also represented by a rational function of $x$ and $\{e^{ik_jx}\}_{j=1}^n$.
This means the desired result.
 \end{proof}
 
\begin{rmk}
We show that
\begin{equation}
u(x)=\frac{1}{\pi}\frac{\partial}{\partial x}\biggl(\sum_{j=1}^n\int_{|\kappa-k_j|=\delta}
 \frac{M(x;\kappa)}{a(\kappa)}\d\kappa\biggr)
\label{eqn:u1}
\end{equation}
$($see, e.g., Section~$9.3$ of {\rm\cite{A11})},
 so that $u(x)$ is also represented by rational functions of $x$ and $\{e^{ik_jx}\}_{j=1}^n$.
The sign of the corresponding formula {\rm(4.12)} in {\rm\cite{Y23}} contains an error.
Note that for example, when $\nu_j=1$,
\[
\frac{1}{\pi}\int_{|\kappa-k_j|}\frac{M(x;\kappa)}{a(\kappa)}\d\kappa
 =\frac{2ib(k_j)}{a_k(k_j)}N_j^0(x),
\]
and when $\nu_j=2$,
\begin{align*}
&
\frac{1}{\pi}\int_{|\kappa-k_j|}\frac{M(x;\kappa)}{a(\kappa)}\d\kappa\\
&
=\frac{4i}{a_{kk}(k_j)}\left(
 b(k_j)N_j^1(x)+\left(b_k(k_j)-\frac{a_{kkk}(k_j)b(k_j)}{3a_{kk}(k_j)}\right)N_j^0(x)\right).
\end{align*}
Since $\Im k_j>0$, $j=1,\ldots,n$,
 it follows from \eqref{eqn:MN1a} and \eqref{eqn:u1} that
\[
\lim_{x\to\pm\infty}u(x)=0.
\]
\end{rmk}

\begin{figure}
\includegraphics[scale=0.6]{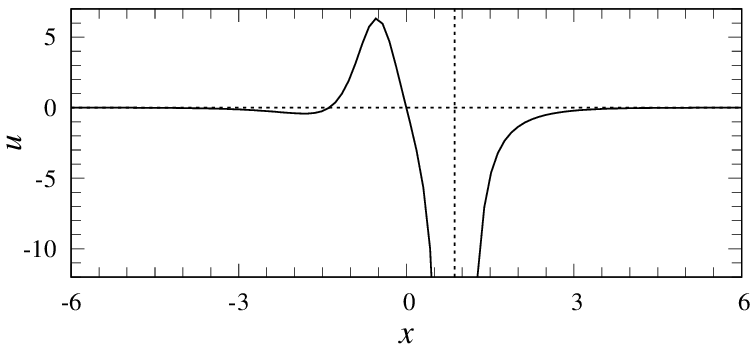}
\caption{Potential $u(x)$ in Example~\ref{ex:3a}.
\label{fig:3a}}
\end{figure}

\begin{ex}
\label{ex:3a}
Let $n=1$ and $\nu_1=2$ in the proof of Theorem~$\ref{thm:main1}$.
For simplicity we assume that $k_1=i$, $a_{kkk}(k_1),b_k(k_1)=0$ and
\begin{equation}
a_{kk}(k_1)=-\tfrac{1}{2},\quad
b(k_1)=1.
\label{eqn:ex3a}
\end{equation}
The system of linear equations about $N_1^0(x),N_1^1(x)$ becomes
\begin{align*}
&
N_1^0(x)e^{2x}
 =1-2iN_1^1(x)+N_1^0(x),\\
&
(N_1^1(x)-2ixN_1^0(x))e^{2x}
=N_1^1(x)+iN_1^0(x).
\end{align*}
Solving the above equation, we obtain
\begin{align*}
N_1^0(x)=\frac{e^{2x}-1}{e^{4x}-2(2x+1)e^{2x}-1},\quad
N_1^1(x)=\frac{i(2x e^{2x}+1)}{e^{4x}-2(2x+1)e^{2x}-1}.
\end{align*}
From \eqref{eqn:MN1} and \eqref{eqn:MN1d} we have
\begin{equation}
\psi(x;k) =N(x;k)e^{-ikx}
 =\biggl(1+\frac{4i[(2(k+i)x+i)e^{2x}+k]}{(k+i)^2(e^{4x}-2(2x+1)e^{2x}-1)}
\biggr)e^{ikx}.
\label{eqn:ex3apsi}
\end{equation}
Moreover, it follows from \eqref{eqn:u1} that
\begin{equation}
u(x)=-\frac{16e^{2x}(e^{2x}-1)((2x-1)e^{2x}+2x+3)}{(e^{4x}-2(2x+1)e^{2x}-1)^2}.
\label{eqn:ex3au}
\end{equation}
See Fig.~$\ref{fig:3a}$ for the shape of $u(x)$.
It has a pole of order two at $x=0.864558\ldots$.
The potential \eqref{eqn:ex3au} is of a negaton-type
 {\rm\cite{M02,MY04,Mat02,RSK96}}.
 
Let $k\in\Rset^\ast$.
We see that
\[
\psi(x;k)\sim\left(\frac{k-i}{k+i}\right)^2e^{ikx}
\]
as $x\to-\infty$.
Hence, it follows from \eqref{eqn:bc} that
\[
\phi(x;k)=\left(\frac{k-i}{k+i}\right)^2\psi(x;-k),
\]
which yields
\[
a(k)=\left(\frac{k-i}{k+i}\right)^2
\]
by \eqref{eqn:ab+}.
Thus, $a_k(i)=0$ and $a_{kk}(i)=-\tfrac{1}{2}$.
On the other hand, using \eqref{eqn:ex3apsi}, we see that $\psi(x;i)\sim e^x$
as $x\to-\infty$, and consequently $b(i)=1$.
These agree with our assumption \eqref{eqn:ex3a}.
\end{ex}


\section{Proof of Theorem~\ref{thm:main2}}
In this section we give a proof of Theorem~\ref{thm:main2}.
Henceforth we assume that the hypotheses of Theorem~\ref{thm:main2} hold
 as well as $k\in\Rset^\ast$.

Letting $y=1/x$, we rewrite \eqref{eqn:ZS0} as
\begin{equation}
\frac{\d^2v}{\d y^2}+\frac{2}{y}\frac{\d v}{\d y}+\frac{k^2-\tilde{u}(y)}{y^4}v=0,
\label{eqn:v}
\end{equation}
where $\tilde{u}(y)=u(1/y)$.
Since $(k^2-\tilde{u}(y))/y^4$ has a pole of fourth order at $y=0$,
 we see by a standard result on higher-order scalar linear differential equations
 (see, e.g., Theorem~7 of \cite{B00}) that
 Eq.~\eqref{eqn:ZS0} has an irregular singularity at infinity $x=\infty$.
To obtain a necessary condition for its integrabilty,
 we compute its formal monodromy, exponential torus and Stokes matrices
  around $y=0$ for \eqref{eqn:v}.
See Appendix~A for necessary information,
 and see also \cite{MS16,PS03,S09} for more details.

Letting $\eta=(\d v/\d y)/v$ in \eqref{eqn:v}, we have a Riccati equation
\begin{equation}
\frac{\d\eta}{\d y}+\eta^2+\frac{2}{y}\eta+\frac{k^2-\tilde{u}(y)}{y^4}=0.
\label{eqn:eta}
\end{equation}
To obtain a formal solution to \eqref{eqn:eta}, we write
\[
\eta=\sum_{j=0}^\infty \eta_jy^{j-2}.
\]
Substituting the above expression into \eqref{eqn:eta}, we have
\[
\eta_0^2+k^2=0,\quad
2\eta_0\eta_1=0,\quad
2\eta_0\eta_j+r_j(\eta_0,\ldots,\eta_{j-1})=0,\quad
j\ge 2,
\]
where $r_j(\eta_0,\ldots,\eta_{j-1})$ is a polynomial of $\eta_0,\ldots,\eta_{j-1}$ for $j\ge 2$.
Thus, we obtain formal solutions to \eqref{eqn:eta} as
\[
\eta(y)=\pm\frac{ik}{y^2}+\cdots
\]
and to \eqref{eqn:v} as
\begin{equation}
v(y)=\exp\left(\mp\frac{ik}{y}\right)v_\pm(y),
\label{eqn:solv}
\end{equation}
where $v_\pm(y)$ are formal power series in $y$,
 and the upper or lower signs are taken simultaneously.
Hence, we express have a formal fundamental matrix of \eqref{eqn:v} as
\begin{equation}
V(y)=\tilde{V}(y)e^{Q(y)},
\label{eqn:V}
\end{equation}
where $\tilde{V}(y)$ is a $2\times 2$ formal meromorphic invertible matrix of power series in $y$ and 
\begin{align*}
Q(y)=
\begin{pmatrix}
-ik/y & 0\\
0 & ik/y
\end{pmatrix}.
\end{align*}

\begin{figure}
\includegraphics[scale=0.6]{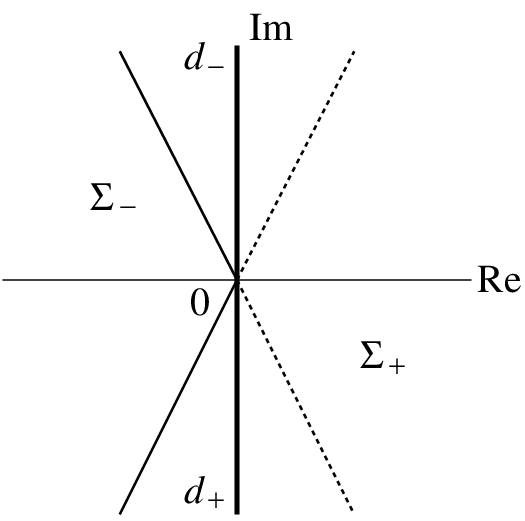}
\caption{Singular directions $d_\pm$ and sectors $\Sigma_\pm$:
The boundaries of $\Sigma_+$ and $\Sigma_-$ are plotted as solid and dotted lines..
\label{fig:1}}
\end{figure}

Substitution of $ye^{2\pi i}$ into \eqref{eqn:V} yields
\[
V(ye^{i\theta})=\tilde{V}(ye^{2\pi i})e^{Q(ye^{2\pi i})},
\]
so that the \emph{formal monodromy} becomes
\[
\hat{M}=V(y)^{-1}V(ye^{2\pi i})=\id
\]
by $V(ye^{2\pi i})=V(y)\hat{M}$.
The \emph{exponential torus} is given by
\[
\T=\T:=\mathrm{Gal}\left(\Cset((y))(e^{ik/y},e^{-ik/y})/\Cset((y))\right)
=\left\{
\begin{pmatrix}
c & 0\\
0 & c^{-1}
\end{pmatrix}
\bigg| c\in\Cset^\ast\right\}.
\]
Here $\mathrm{Gal}(\Lset/\Kset)$ denotes the differential Galois group,
 where $\Kset$ and $\Lset$ are differential fields
 such that $\Lset$ is a differential field extension of  $\Kset$.
The \emph{Stokes directions},
 on which $\Re(\mp ik/y)=0$, are given by $\arg y=0,\pi$.
We have $\Re(-ik/y)<0$ and  $\Re(ik/y)<0$ for $\arg y\in(0,\pi)$ and $(\pi,2\pi)$, respectively.
So the \emph{singular directions} $d_\mp$,
 which are bisectors of $(\tfrac{1}{2}\pi\mp\tfrac{1}{2}\pi,\tfrac{3}{2}\pi\mp\tfrac{1}{2}\pi)$,
 become $\arg y=\pi\mp\tfrac{1}{2}\pi$.
See Fig.~\ref{fig:1}.
By Proposition ~\ref{prop:a1}, we can write the Stokes matrices as
\begin{equation}
S_-=\begin{pmatrix}
1 & \alpha_-\\
0 & 1
\end{pmatrix}
\quad\mbox{and}\quad
S_+=\begin{pmatrix}
1 & 0\\
\alpha_+ & 1
\end{pmatrix}
\label{eqn:Stokes}
\end{equation}
for the singular directions $d_-$ and $d_+$, respectively,
 where $\alpha_\pm$ are constants which may be zero.

Suppose that $a(k),a(-k)\neq 0$.
Let
\begin{align*}
&
\Phi(y;k)=\begin{pmatrix}
\phi(y^{-1};k) & \phi(y^{-1};-k)\\
-y^{-2}\phi'(y^{-1};k) & -y^{-2}\phi'(y^{-1};-k)
\end{pmatrix},\\
&
\Psi(y;k)=\begin{pmatrix}
\psi(y^{-1};-k) & \psi(y^{-1};k)\\
-y^{-2}\psi'(y^{-1};-k) & -y^{-2}\psi'(y^{-1};k)
\end{pmatrix},
\end{align*}
which are fundamental matrices to \eqref{eqn:v}.
By \eqref{eqn:ab+} and \eqref{eqn:ab-},
\begin{equation}
\Phi(y;k)=\Psi(y;k)B(k),\quad
B(k)=\begin{pmatrix}
a(k) & b(-k)\\
b(k) & a(-k)
\end{pmatrix}.
\label{eqn:thm2}
\end{equation}
 
\begin{lem}
\label{lem:4a}
If $\alpha_-$ and $\alpha_+=0$, then $b(-k)$ and $b(k)=0$, respectively.
\end{lem}

\begin{proof}
We prove that if $\alpha_-=0$, then $b(-k)=0$.
The remaining part can be proven similarly.

First, by \eqref{eqn:bc}, we can write the formal fundamental matrix as
\[
V(y)=\left(\begin{array}{cc}
\begin{matrix}
*\\
*
\end{matrix}\quad
\Phi(y;k)\begin{pmatrix}
0\\
c_-
\end{pmatrix}
\end{array}\right)
\quad\mbox{and}\quad
\left(\begin{array}{cc}
\begin{matrix}
*\\
*
\end{matrix}\quad
\Psi(y;k)\begin{pmatrix}
0\\
c_+
\end{pmatrix}
\end{array}\right)
\]
in the directions $d_\l$ and $d_\r$ near $d_-$ in $\Sigma_-$ and $\Sigma_+$, respectively,
 where $c_\pm\in\Cset$ are nonzero constants.
If $\alpha_-=0$, then we use \eqref{eqn:Stokes} and  \eqref{eqn:thm2} to obtain
\[
b(-k)c_-=0,\quad
a(-k)c_--c_+=0,
\]
by the definition of the Stokes matrix.
Thus, we obtain the desired result.
\end{proof}

Theorem~\ref{thm:a1} implies that the differential Galois group $G$ of \eqref{eqn:ZS0} contains
 the Zariski closure of a group generated by the formal monodromy $\hat{M}=\id$,
 the exponential torus $\T$ and the Stokes matrices $S_\pm$.
In particular, $G$ is not finite,
 i.e., cases (i) and (ii) of Proposition~\ref{prop:2b} do not occur.
Moreover, case~(iv) of Proposition~\ref{prop:2b} does not occur
 since $a(k)\neq 0$.
Hence, if Eq.~\eqref{eqn:ZS0} is integrable in the sense of differential Galois theory,
 then we have have $\alpha_-$ or $\alpha_+=0$, so that by Lemma~\ref{lem:4a} $b(k)$ or $b(-k)=0$,
 since either case~(iii) or (v) of Proposition~\ref{prop:2b} occurs.
So we have the following.

\begin{lem}
\label{prop:4a}
Let $k\in\Rset^\ast$ and assume that $a(k)\neq 0$.
If Eq.~\eqref{eqn:ZS0} is solvable by quadrature, then $b(k)$ or $b(-k)=0$.
\end{lem}

\begin{proof}[Proof of Theorem~$\ref{thm:main2}$]
Suppose that Eq.~\eqref{eqn:ZS0} is solvable by quadrature.
Since by Proposition~\ref{prop:2a}(iii) $a(k)=0$ only at discrete points,
 it follows from Lemma~\ref{prop:4a} that $b(k)$ or $b(-k)=0$ except at the discrete points.
Hence, by the identity theorem (e.g., Theorem~3.2.6 of \cite{AF03}),
  we have $b(k)=0$ for any $k\in\Rset^\ast$
  since by Proposition~\ref{prop:2a}(i) $b(k)$ is analytic
  in a neighborhood of $\Rset^\ast$ in $\Cset^\ast$.
Thus, we complete the proof.
\end{proof}


\section{Proof of Theorem~\ref{thm:main3}}

In this section we give a proof of Theorem~\ref{thm:main3}.
The proof is mainly relied to the result of Kovacic \cite{K86}, which is briefly described in Appendix~B.

\begin{proof}[Proof of Theorem~$\ref{thm:main3}$]
Henceforth we assume that $m_2-m_1=1$.
Obviously, case (c) of Proposition~\ref{prop:K2} does not occur for $k\in\Rset^\ast$. 
So we only have to show that cases (a) and (b) of Proposition~\ref{prop:K1} do not occur
 for some $k\in\Rset^\ast$.
We rewrite \eqref{eqn:ZS0} in the form \eqref{eqn:req} with
\begin{equation}
r(x)=-k^2-u(x)
\label{eqn:sec5a}
\end{equation}
and expand $u(x)$ in the Laurent series at $x=\infty$ as
\[
u(x)=\frac{\bar{u}_1}{x}+\cdots,
\]
where $\bar{u}_1\neq 0$ is a constant.
Let $\S$ denote the set of poles of $u(x)$.

We assume that case (a) of Proposition~\ref{prop:K1} occurs
 as well as $k\in\Rset^\ast$.
Following the recipe of Appendix~B2,
 we have $d_c\in\Nset_0:=\Nset\cup\{0\}$ in \eqref{eqn:dc} and 
\[
\kappa_\infty^\pm=\pm\frac{i\bar{u}_1}{2k}
\]
since $\alpha=ik$ and $\beta=-\bar{u}_1$.
We easily see that $d_c\not\in\Nset_0$ when $k$ changes slightly
 even if $d_c\in\Nset_0$for some $k\in\Rset^\ast$ .
This implies that case~(a) of Proposition~\ref{prop:K1} does not occur
 in an open interval of $k\in\Rset^\ast$ since $\S$ is finite
 so that the number of a finite sequence $c$ of the symbols $\pm$ with $d_c\in\Nset_0$ is finite.

We next assume that case (b) of Proposition~\ref{prop:K1} occurs
 for an open interval of $k\in\Rset^\ast$.
Following the recipe of Appendix~B3,
 we have $d_e\in\Nset_0$ in \eqref{eqn:de}
 and there exists a monic polynomial $P(x)$ of order $d_e$ satisfying \eqref{eqn:P}.
Since
\[
\lim_{x\to\pm\infty}
(\theta''(x)+3\theta(x)\theta'(x)+\theta(x)^3+4(k^2+u(x))\theta(x)+2u'(x))=4k^2,
\]
we easily see that $d_e\neq 0$.
Let $d=d_e\ge 1$.
Since Eq.~\eqref{eqn:P} holds, we see that
\begin{equation}
\begin{split}
&
(\theta''(x)+3\theta(x)\theta'(x)+\theta(x)^3+4(k^2+u(x))\theta(x)+2u'(x)),\\
&
(\theta''(x)+3\theta(x)\theta'(x)+\theta(x)^3+4(k^2+u(x))\theta(x)+2u'(x))x\\
&\qquad
+(3\theta(x)^2+3\theta'(x)+4(k^2+u(x))),\\
&
(\theta''(x)+3\theta(x)\theta'(x)+\theta(x)^3+4(k^2+u(x))\theta(x)+2u'(x))x^2\\
&\quad
+2(3\theta(x)^2+3\theta'(x)+4(k^2+u(x)))x+6\theta(x),\\
&
(\theta''(x)+3\theta(x)\theta'(x)+\theta(x)^3+4(k^2+u(x))\theta(x)+2u'(x))x^j\\
&\quad
+j(3\theta(x)^2+3\theta'(x)+4(k^2+u(x)))x^{j-1}+3j(j-1)\theta(x)x^{j-2}\\
&\quad
+j(j-1)(j-2)x^{j-3},\quad
j=3,\ldots,d,
\end{split}
\label{eqn:prop2c1}
\end{equation}
are linearly dependent, on the interval of $k$,
 where the last equation is eliminated if $d=2$,
 and the last two equations are eliminated if $d=1$.
Hence,
\[
\theta(x),\quad
(\theta(x)x+j)x^{j-1},\quad
j=1,\ldots,d,
\]
are linearly dependent.
If $d=1$, then
\[
(\theta(x)x+1)+q_0\theta(x)
=\theta(x)(x+q_0)+1=0.
\]
If $d\ge 2$, then
\[
\theta(x),\quad
(\theta(x)x+j)x^{j-1},\quad
j=1,\ldots,d-1,
\]
are linearly independent, so that for certain constants $q_j$, $j=0,\ldots,d-1$,
\begin{align*}
&
(\theta(x)x+d)x^{d-1}+\sum_{j=1}^{d-1}q_j(\theta(x)x+j)x^{j-1}+q_0\theta(x)\\
&=\theta(x)\sum_{j=0}^d q_jx^j+\sum_{j=1}^d jq_jx^{j-1}=0,
\end{align*}
where $q_d=1$.
Thus, we have
\begin{equation}
\theta(x)=-\frac{q'(x)}{q(x)},\quad
q(x)=\sum_{j=0}^d q_jx^j,
\label{eqn:prop2c2}
\end{equation}
and
\begin{equation}
P(x)=\sum_{j=0}^dq_jx^j=q(x),
\label{eqn:prop2c3}
\end{equation}
since the functions given by \eqref{eqn:prop2c1}
 are linearly dependent on the interval of $k$.
Substituting \eqref{eqn:prop2c2} and \eqref{eqn:prop2c3} into \eqref{eqn:P},
 we obtain
 \[
q(x)u'(x)+4q'(x)u(x)=0,
 \]
which yields
\[
u(x)=Cq(x)^{-4},
\]
where $C\neq 0$ is a constant.
This means that $m_2-m_1=-4$ and yields a contradiction.
Thus we complete the proof.
\end{proof}

\begin{rmk}
In the above proof,
 an essential role was played
 by the fact that poles of $r(x)$ given by \eqref{eqn:sec5a} do not depend on $k$.
\end{rmk}

\begin{ex}
\label{ex:2b}
Consider the rational potential
\[
u(x)=-\frac{5}{16x^2}-\frac{5}{16(x-1)^2}-\frac{7}{8x}+\frac{5}{24(x-1)}
\]
and let $r(x)=-k^2-u(x)$.
Equation~\eqref{eqn:ZS0} satisfies $m_2-m_1=1$ 
 and only has singularities of order two at $x=0$ and $1$.
Following the recipe of Appendix~{\rm B.2}, we have $\beta=5/16$ and
\[
\kappa_s^\pm=\tfrac{1}{2}\pm\tfrac{3}{4}
\]
for $s=0,1$, and
\[
\kappa_\infty^\pm=\mp\frac{19i}{48k}
\]
since
\[
\bar{u}_1=\lim_{x\to\infty}xu(x)=-\tfrac{19}{24}.
\]
Hence, case {\rm(a)} of Proposition~{\rm\ref{prop:K1}} does not occur
 for any $k\in\Rset^\ast$.
Following the recipe of Appendix~{\rm B.3}, we have $\beta=5/16$ and
\[
E_s=\{-1,2,5\}
\]
for $s=0,1$.
We see that $d_e=1$ for $e_0,e_1=-1$ is only a non-negative integer, so that
\[
\theta(x)=-\frac{1}{2x}-\frac{1}{2(x-1)}.
\]
Substituting $P(x)=x+p_0$ into \eqref{eqn:P}, we obtain
\begin{align*}
&
\frac{1}{x^2(x-1)^3}
\bigl(-(k^2p_0+2k^2-\tfrac{4}{3})x^4+(2(5k^2+2)p_0+4k^2-1)x^3\\
&
-\tfrac{1}{3}(8(3k^2-5)p_0+2(3k^2+5))x^2
+\tfrac{1}{3}(2(3k^2+20)p_0+4)x+4p_0+1\bigr)=0.
\end{align*}
Hence, case~{\rm(b)} of Proposition~{\rm\ref{prop:K1}} holds with $p_0=-\tfrac{1}{4}$,
 so that Eq.~\eqref{eqn:ZS0} is solvable by quadrature,
 if and only if $k=\pm\tfrac{2}{3}\sqrt{3}$.

Let $k=\pm\tfrac{2}{3}\sqrt{3}$.
We have
\[
\hat{\theta}(x)=-\frac{1}{2x}-\frac{1}{2(x-1)}+\frac{4}{4x-1},
\]
so that the integrand in \eqref{eqn:solreq} satisfies
\[
\tfrac{1}{2}\left(-\hat{\theta}(x)
 \pm\sqrt{-4(k^2+u(x))-\hat{\theta}(x)^2-2\hat{\theta}'(x)}\right)
 \to \pm ik\left(1-\frac{19}{48k^2x}\right)+O(x^{-2})
\]
as $|x|\to\infty$ on $\Rset$.
This implies that Eq.~\eqref{eqn:ZS0} has two linearly independent solutions such that
\[
v(x)\to x^{-19i/48k}e^{ik x}\mbox{ and }x^{19i/48k}e^{-ikx}\quad\mbox{as $x\to\pm\infty$}.
\]
Thus, Eq.~\eqref{eqn:ZS0} does not have the Jost solutions satisfying \eqref{eqn:bc}
 even when it is solvable by quadrature.
\end{ex}

\section*{Acknowledgements}
This work was partially supported by the JSPS KAKENHI Grant Number JP22H01138.

\section*{Data Availability}
Data sharing not applicable to this article
 as no datasets were generated or analyzed during the current study.
 
 \section*{Conflict of Interest Statement}
 The author has no conflict of interest to disclose.


\appendix
\renewcommand{\theequation}{\Alph{section}.\arabic{equation}}

\section{Local Differential Galois Groups near Irregular Singularities}
 
In this appendix we review basic information on local differential Galois groups,
 which consist of formal monodromies, exponential tori and Stokes matrices,
 near irregularity singular points.
See \cite{MS16,PS03,S09} for the details.
Recall some terminologies of the differential Galois theory in Section~3 of \cite{Y23}
  (see also \cite{CH11,PS03}) if necessary in the following.

Consider $n$-dimensional linear systems of the form
\begin{equation}
\frac{\d\zeta}{\d y}=A(y)\zeta
\label{eqn:eta1}
\end{equation}
with an irregular singularity at $y=0$ in $\Cset$, where $A(y)\in\gl(n,\Cset((y)))$,
 and $\Cset((y))$ denotes the field of the formal Laurent series over $\Cset$.
In a standard way,
 Eq.~\eqref{eqn:v} can be written in the form \eqref{eqn:eta1} easily.
The linear system~\eqref{eqn:eta1} has a formal fundamental matrix of the form
\begin{equation}
Y(y)=\hat{Y}(y)y^Le^{Q(1/z)}
\label{eqn:Y1}
\end{equation}
like \eqref{eqn:V}, where
\begin{itemize}
\setlength{\leftskip}{-2.4em}
\item
$y^\ell=z$ for some $\ell=1,\ldots,n$;
\item
$\hat{Y}(y)\in\GL(n;\Cset((y)))$;
\item
$L\in\gl(n;\Cset)$;
\item
$Q(1/z)=\diag(q_1(z),\ldots,q_n(z))$, $q_j(z)\in\frac{1}{z}\Cset[\frac{1}{z}]$
\end{itemize}
(see, e.g., \cite{L01}).
So we obtain a formal Picard-Vessiot extension $\Lset\subset\Kset=\Cset((y))$,
 where $\Lset$ is a differential field generated by entries of $Y(y)$
 and having $\Cset$ as its field of constants.

Let $\hat{\mu}$ be a $\Kset$-automorphism of $\Lset$ satisfying
\begin{itemize}
\setlength{\leftskip}{-2.4em}
\item
$\hat{\mu}(y^{\alpha_j})=e^{2\pi i\alpha_j}y^{\alpha_j}$ for any $j=1,\ldots,n$;
\item
$\hat{\mu}(\log y)=\log y+2\pi i$;
\item
$\hat{\mu}(\exp(q_j(y^{1/\ell})))=\exp(q_j(e^{2\pi i/\ell}y^{1/\ell})$,
\end{itemize}
where $\alpha_j$, $j=1,\ldots,n$, denote the eigenvalues of $L$.
Since $\hat{\mu}$ commutes the differentiation,
 $\hat{\mu}(Y(y))$ is also a formal fundamental matrix.
Hence, there exists a matrix $\hat{M}\in\GL(n;\Cset)$ such that
\[
\hat{\mu}(Y(y))=Y(y)\hat{M}.
\]
We call $\hat{M}$ the \emph{formal monodromy} of \eqref{eqn:eta1}.
Thus, $\hat{M}$ can be computed by substituting $ye^{i\theta}$ into $Y(y)$, as in Section~4.

Let $\Eset$ be a differential field extension of $\Kset$
 such that $\Lset=\Eset(e^{q_1(y^{1/\ell})},\ldots,e^{q_n(y^{1/\ell})})$.
Then $\T=\mathrm{Gal}(\Lset/\Eset)$
 is called the \emph{exponential torus} of \eqref{eqn:eta1}.
Thus, $\T$ is the group of all $\Eset$-automorphisms of $\Lset$.

We next state the definition of the \emph{Stokes matrix} for \eqref{eqn:eta}.
Some preparations are needed to this end.
Let $q(z)\in z^{-1}\Cset[z^{-1}]$,
 so that $q(z^{-1})$ is a polynomial of $z$.
Let $k$ be the degree of $q(z^{-1})$
 and let $q(z)=c_0z^{-k}+\ldots$, where $c_0\in\Cset$ is a constant.
The direction $d$ determined by $\Re(c_0y^{-k/\ell})=0$, 
 in which $e^{q(z)}$ changes its behavior, 
 is called the \emph{Stokes direction} for $q(z)$.
Let $d_1,d_2$ be consecutive Stokes directions.
We say that $(d_1,d_2)$ is a \emph{negative Stokes pair}
 if $\Re(c_0y^{-k/\ell})<0$ for $\arg(y)\in(d_1,d_2)$.
We refer to the bisector a negative Stokes pair
 as a \emph{singular direction} of $q(z)$.

\begin{figure}
\includegraphics[scale=0.75]{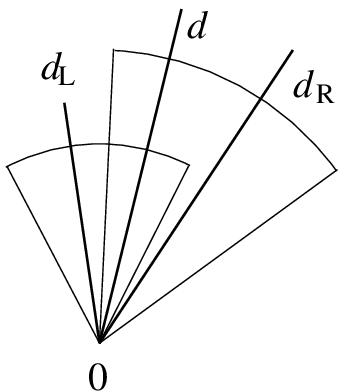}
\caption{Singular direction $d$ and the two directions $d_\l$ and $d_\r$.
\label{fig:a}}
\end{figure}

We now consider the formal fundamental matrix \eqref{eqn:Y1}.
Let $d$ be a singular direction of $q_j(z)-q_l(z)$ for some $j,l=1,\ldots,n$ with $j\neq l$.
Choose two directions $d_\l$ and $d_\r$ counterclockwise nearby
 such that the associated sectors overlap and contain $d$.
See Fig.~\ref{fig:a}.
Let $\hat{Y}_\l(y)$ and $\hat{Y}_\r(y)$ be the multisums  of $\hat{Y}(y)$
 in the directions $d_\l$ and $d_\r$, respectively.
Then there exists $S_{jl}\in\GL(n;\Cset)$ such that
\begin{equation}
\hat{Y}_\l(y)y^Le^{Q(1/z)}=\hat{Y}_\r(y)y^Le^{Q(1/z)}S_{jl}
\label{eqn:app}
\end{equation}
since Eq.~\eqref{eqn:Y1} gives a formal fundamental matrix
 for both cases.
The matrix $S_{jl}$ is independent of the choice of  $d_\l$ and $d_\r$
 and called a \emph{Stokes matrix} for the singular direction $d$.
We have the following on the Stokes matrices
 (see, e.g., Proposition~2.37 of \cite{MS16} for its proofs).

\begin{prop}
\label{prop:a1}
The Stokes matrices are unipotent such that all the diagonal elements are one
 and the $(j',l')$-element is zero if $d$ is a singular direction for $q_{j'}(z)-q_{l'}(z)$.
\end{prop}

Finally, we state one of the most important results on local differential Galois groups
 near irregular singularities as follows.
  
\begin{thm}[Ramis \cite{R85,R07}]
\label{thm:a1}
The local differential Galois group of \eqref{eqn:eta} around the irregular singular point $y=0$
  is the Zariski closure of the group generated the formal monodromy $\hat{M}$,
  the exponential torus $\T$ and the Stokes matrices $S_{jl}$, $j,l=1,\ldots,n$ with $j\neq l$.
\end{thm}

See also \cite{L94,PS03} for a proof of Theorem~\ref{thm:a1}.
Theorem~\ref{thm:a1} is regarded as a generalization of Schlesinger's theorem (see, e.g., \cite{PS03,Y23})
 for regular singularities to irregular ones.

\section{Kovacic's Result}

In this appendix we review necessary parts of Kovacic \cite{K86}
 on second-order linear differential equations on $\Cset(x)$.
 
\subsection{Fundamental results}
Consider second-order linear differential equations of the form
\begin{equation}
v_{xx}=r(x)v,
\label{eqn:req}
\end{equation}
where $r(x)\in\Cset(x)\setminus\Cset$.
Let $G$ be the differential Galois group of \eqref{eqn:req} over $\Cset(x)$.
We easily see that $G$ is an algebraic subgroup of $\SL(2,\Cset)$
 (see, e.g., Chapter~6 of \cite{K76}).
We have the following four cases for \eqref{eqn:req} (see \cite{CH11,K86} for a proof).

\begin{prop}[Kaplansky \cite{K76}, Kovacic\cite{K86}]
\label{prop:K1}
One of the following four cases occurs$\,:$
\begin{enumerate}
\setlength{\leftskip}{-1.8em}
\item[(a)]
$G$ is triangularizable$\,;$
\item[(b)]
$G$ is conjugate to a subgroup of
\[
D^\dagger=\left\{
\begin{pmatrix}
c & 0\\
0 & c^{-1}
\end{pmatrix}
\bigg|
\,c\in\Cset^\ast\right\}\cup
\left\{
\begin{pmatrix}
0 & c\\
-c^{-1} & 0
\end{pmatrix}
\bigg|
\,c\in\Cset^\ast\right\},
\]
but it is not triangularizable$\,;$
\item[(c)]
$G$ is finite, but neither triangulariable nor conjugate to $D^\dagger;$ 
\item[(d)]
$G=\SL(2,\Cset)$.
\end{enumerate}
\end{prop}

If one of cases~(a)-(c) of Proposition~\ref{prop:K1} occurs,
 then the connected identity component $G^0$ of $G$ is solvable
 and Eq.~\eqref{eqn:req} is solvable.
On the other hand, if cases~(a)-(c) do not occur,
 then $G=G^0=\SL(2,\Cset)$, i.e.,
 Eq.~\eqref{eqn:req} is not integrable in the sence of  differential Galois theory.

Let $r(x)=r_1(x)/r_2(x)$, where $r_1(x),r_2(x)\in\Cset[z]$ are relatively prime.
Poles of $r(x)$ are zeros of $r_2(x)$
 and their orders are the multiplicities of the zeros of $r_2(x)$.
Let $\S\subset\Cset$ denote the set of all poles of $r(x)$.
By the \emph{order of $r(x)$ at $\infty$},
 we also mean the order of $\infty$ as a zero of $r(x)$, i.e., $\deg r_2(x)-\deg r_1(x)$.
We have the following necessary conditions for cases (a)-(c) in Proposition~\ref{prop:K1} to occur
 (see \cite{CH11,K86} for a proof).

\begin{prop}[Kovacic \cite{K86}]
\label{prop:K2}
The following conditions are necessary
 for the corresponding cases in Proposition~$\ref{prop:K1}$ to occur$\,:$
\begin{itemize}
\setlength{\leftskip}{-1.8em}
\item[(a)]
No pole of $r(x)$ is of odd order greater than one,
 and the order of $r(x)$ at $\infty$ is not an odd less than two$\,;$
\item[(b)]
There exists a pole of $r(x)$ such that its order is two or an odd greater than two$\,;$
\item[(c)]
No pole of $r(x)$ is of order greater than two
 and the order of $r(x)$ at $\infty$ is greater than one.
\end{itemize}
\end{prop}

Note that cases~(a), (b) and (c) of Proposition~\ref{prop:K1} may not occur
 even though conditions~(a), (b) and (c) hold, respectively.
In particular, Proposition~\ref{prop:K2} only states fewer conditions for case~(c)
 than the original version in \cite{K86} since they are enough for our application
 in Section~5.
Kovacic \cite{K86} also gave an algorithm for finding a ``closed-form'' solution
 and detecting whether cases (a)-(c) in Proposition~\ref{prop:K1}
 really occur or not under conditions~(a), (b) and (c) of Proposition~\ref{prop:K1},
 respectively.
In the following
 we describe necessary parts of his algorithm for cases~(a) and (b)
 in the setting required in Section~5:
The order of $r(x)$ at $\infty$ is zero.
See \cite{CH11,K86} for the full version of his algorithm and its proof.

\subsection{Case (a)}
We assume that condition~(a) in Proposition~\ref{prop:K2} holds
 as well as the order of $r(x)$ at $\infty$ is zero.
When $g(x)$ has a pole of order $\ell\ge 2$ at $x=s$ and is expanded in a Laurent series
\[
g(x)=\sum_{j\ge-\ell}g_j(x-s)^j
\]
there, we write
\[
[g(x)]_s=\sum_{-\ell\le j\le -2}g_j(x-s)^j.
\]
Similarly, when $g(x)$ is of order $\ell$ at $\infty$ and is expanded in a Laurent series
\[
g(x)=\sum_{j\le\ell}g_jx^j
\]
there, we write
\[
[g(x)]_\infty=\sum_{0\le j\le\ell}g_jx^j.
\]
For $s\in\S\cup\{\infty\}$
 we define a rational function $\omega_s(x)$ and two constants $\kappa_s^\pm$ as follows:
\begin{enumerate}
\setlength{\leftskip}{-1.2em}
\setlength{\labelsep}{1em}
\item[($s_1$)]
If $s\in\S$ is of order $1$, then $\kappa_s^\pm=1$;
\item[($s_2$)]
If $s\in\S$ is of order $2$, then 
\[
\kappa_s^\pm=\tfrac{1}{2}\pm\tfrac{1}{2}\sqrt{1+4\beta},
\]
where $\beta$ is the coefficient of $1/(x-s)^2$ in the partial fraction expansion of $r(x);$
\item[($s_3$)]
If $s\in\S$ is of order $2\ell>2$, then
\[
\kappa_s^\pm=\tfrac{1}{2}\left(\pm\frac{\beta}{\alpha}+\ell\right),
\]
where $\alpha$ and $\beta$ are the coefficients of $(x-s)^{-\ell}$ and $(x-s)^{-(\ell+1)}$
 in the Laurent series expansions of $\sqrt{r(x)}$ and $r(x)-([\sqrt{r(x)}]_s)^2$, respectively;
\end{enumerate}
\begin{enumerate}
\setlength{\leftskip}{-0.7em}
\setlength{\labelsep}{1em}
\item[($\infty_1$)]
$\omega_\infty(x)=[\sqrt{r(x)}]_\infty$ and
\[
\kappa_\infty^\pm=\pm\frac{\beta}{2\alpha},
\]
where $\alpha$ and $\beta$ are the coefficients of $x^0$ and $x^{-1}$
 in the Laurent series expansions of $\sqrt{r(x)}$ and $r(x)$, respectively.
\end{enumerate}

For each sequence $c=\{c(s)\}_{s\in\S\cup\{\infty\}}$ with $c(s)=+$ or $-$, define
\begin{equation}
d_c=\kappa_\infty^{c(\infty)}-\sum_{s\in\S}\kappa_s^{c(s)}.
\label{eqn:dc}
\end{equation}
If $d_c$ is not a non-negative integer for any sequence $c$,
 then case~(a) of Proposition~\ref{prop:K1} does not occur.
We also have a further recipe to determine
 whether it occurs or not when $d_c$ is a non-negative integer,
 although it is not given here since it is not necessary for our purpose.
The reader who is interested in it should consult \cite{CH11,K86}.

\subsection{Case (b)}
We assume that condition~(b) in Proposition~\ref{prop:K2} holds
 as well as the order of $r(x)$ at $\infty$ is zero.
For $s\in\S\cup\{\infty\}$ we define a set $E_s\subset\Zset$ as follows:
\begin{enumerate}
\setlength{\leftskip}{-1.2em}
\setlength{\labelsep}{1em}
\item[($s_1$)]
If $s\in\S$ is of order $1$, then $E_s=\{4\}$;
\item[($s_2$)]
If $s\in\S$ is of order $2$, then
\[
E_s=\{2+l\sqrt{1+4\beta}\mid l=0,\pm 2\}\cap\Zset,
\]
where $\beta$ is the coefficient of $1/(x-s)^2$ in the partial fraction expansion of $r\,;$
\item[($s_3$)]
If $s\in\S$ is of order $j>2$, then $E_s=\{j\}$.
\end{enumerate}

For each tuple $e=\{e_s\in E_s\mid s\in\S\}$, let
\begin{equation}
d_e=-\tfrac{1}{2}\sum_{s\in\S}e_s.
\label{eqn:de}
\end{equation}
If $d_e$ is not a non-negative integer for any $e$, 
 then case~(b) of Proposition~\ref{prop:K1} does not occur.

Let $E=\{e\mid d_e\in\Nset_0\}$ and let
\begin{equation}
\theta(x)=\tfrac{1}{2}\sum_{s\in\S}\frac{e_s}{x-s}
\label{eqn:theta1}
\end{equation}
for $e\in E$.
By \eqref{eqn:de} and \eqref{eqn:theta1} we see that
 the Laurent series of $\theta(x)$ at $\infty$ is given by
\[
\theta(x)=-\frac{d_e}{x}+\cdots.
\]
If there exists a monic polynomial $P(x)$ of degree $d_e$ satisfying
\begin{align}
&P'''+3\theta(x)P''+(3\theta(x)^2+3\theta'(x)-4r(x))P'\notag\\
&
+(\theta''(x)+3\theta(x)\theta'(x)+\theta(x)^3-4r(x)\theta(x)-2r'(x))P=0,
\label{eqn:P}
\end{align}
then case~(b) of Proposition~\ref{prop:K1} occurs and
\begin{equation}
v(x)=\exp\left(\int\tfrac{1}{2}\left(-\hat{\theta}(x)
  \pm\sqrt{4r(x)-\hat{\theta}(x)^2-2\hat{\theta}'(x)}\right)\d x\right)
\label{eqn:solreq}
\end{equation}
are solutions to \eqref{eqn:req},
 where $\hat{\theta}(x)=\theta(x)+P'(x)/P(x)$.
If such a polynomial is not found for any sequence $e$,
 then it does not occur.
 

\end{document}